\documentclass{amsart}
\usepackage[utf8]{inputenc}
\usepackage{amsfonts}
\usepackage{amsmath}
\usepackage{amssymb}
\usepackage{amsthm}
\usepackage{hyperref}
\usepackage[english]{babel}
\usepackage{url}
\usepackage{verbatim,listings}
\usepackage{tikz,tikz-cd}
\usepackage[margin=1.0in]{geometry}
\newtheorem{thm}{Theorem}[section]
\newtheorem{lem}[thm]{Lemma}

\newtheorem{con}[thm]{Conjecture}

\theoremstyle{definition}
\newtheorem{defn}[thm]{Definition}

\theoremstyle{remark}
\newtheorem*{rem}{Remark}

\newcommand{\divides}{\,\Big{|}\,}
\newcommand{\inlinemod}[2]{\equiv #1 \,\,(\!\!\!\!\mod #2)}

\title{A Proposed Crank for $(k+j)-$colored partitions, with $j$ colors having distinct parts}
\author{Sam Wilson}
\date{}

\begin{document}

\begin{abstract}
    In 1988, George Andrews and Frank Garvan discovered a crank for $p(n)$. In 2020, Larry Rolen, Zack Tripp, and Ian Wagner generalized the crank for $p(n)$ in order to accommodate Ramanujan-like congruences for $k$-colored partitions. In this paper, we utilize the techniques used by Rolen, Tripp, and Wagner for crank generating functions in order to define a crank generating function for $(k+j)$-colored partitions where $j$ colors have distinct parts. We provide three infinite families of crank generating functions and conjecture a general crank generating function for such partitions.
\end{abstract}

% Keywords: Crank, Integer Partitions, Colored Partitions
% MSC Codes: 05A15, 05A17, 11F33

\maketitle

\section{Introduction}
In 1944 Freeman Dyson defined the \textit{rank} of a partition in order to provide a combinatorial explanation for the famous partition congruences discovered by Ramanujan, 
\begin{align*}
    p(5n+4) &\equiv 0 \mod{5} \\
    p(7n+5) &\equiv 0 \mod{7} \\
    p(11n+6) &\equiv 0 \mod{11} .
\end{align*} 

To calculate the rank, one subtracts the largest part of the partition and its number of parts. The purpose of the rank is to distribute the partitions of a number $\ell n+\delta$ into $\ell$-equinumerous sets. For example, if $n=1$, $\ell=5$, and $\delta=4$, we expect the number of partitions of $9$ to be divisible by $5$. Indeed, Table \ref{table:rank} shows how the rank divides up the partitions of $9$ into $5$ equally sized sets. One may check that the rank always gives such a distribution for $p(5n+4)$ and $p(7n+5)$, however the rank fails to be sufficient for $p(11n+6)$.
In 1944 Dyson conjectured the existence of a generalized rank, which he called the \textit{crank}, that works for all three of Ramanujan's partition congruences.
\begin{table}[h!]
\centering
\begin{tabular}{||c | c | c | c | c ||} 
 \hline
 rank $\inlinemod{0}{5}$ & rank $\inlinemod{1}{5}$  & rank $\inlinemod{2}{5}$  & rank $\inlinemod{3}{5}$ & rank $\inlinemod{4}{5}$  \\ [0.5ex] 
 \hline\hline
 \{2,2,1,1,1,1,1\} & \{2,2,2,1,1,1\} & \{1,1,1,1,1,1,1,1,1\} & \{3,2,2,1,1\} & \{2,1,1,1,1,1,1,1\} \\ 
 \{3,3,3\} & \{3,1,1,1,1,1,1\} & \{2,2,2,2,1\} & \{4,1,1,1,1,1\} & \{3,2,2,\} \\
 \{4,2,2,1\} & \{4,3,2\} & \{3,2,1,1,1,1\} & \{3,3,1,1,1\} & \{3,3,2,1\} \\
 \{4,3,1,1\} & \{4,4,1\} & \{5,2,2\} & \{5,4\} & \{4,2,1,1,1\} \\
 \{5,1,1,1,1\} & \{5,2,1,1\} & \{5,3,1\} & \{6,2,1\} &  \{6,3\}\\ 
 \{7,2\}& \{8,1\} & \{6,1,1,1\}& \{9\} & \{7,1,1\}\\
 \hline
\end{tabular}
\caption{The partitions of $9$ sorted by rank.}
\label{table:rank}
\end{table}
Such a crank was defined by George Andrews and Frank Garvan in 1988 \cite{Andrews1988DysonsCO}. The definition for their crank is given below. 

\begin{defn}\label{defn:rank}
 Let $\lambda$ be an integer partition, $l(\lambda)$ be the largest part, $\omega(\lambda)$ be the number of $1$s in $\lambda$, and $\mu(\lambda)$ be the number of parts in $\lambda$ which are larger than $\omega(\lambda)$. The crank of $\lambda$ is defined to be
 $$
 \begin{cases}
     l(\lambda) & \text{ if } \omega(\lambda) = 0 \\
     \mu(\lambda) - \omega(\lambda) & \text{ if } \omega(\lambda) > 0.
 \end{cases}
 $$
\end{defn}

\begin{rem}
    It should be noted that, for $\ell=5$ or $\ell=7$, although both the rank and the crank (mod $\ell$) distribute the partitions of $n$ into $\ell$ equally sized sets, these sets are not necessarily equal. For example, the crank of the partition $\{ 5, 2, 2 \}$ is congruent to $0$ mod $5$ while the rank of this partition is congruent to $2$ mod $5$.
\end{rem}
Andrews and Garvan also provided the generating function for the crank. Namely, if $M(m,n)$ denotes the number of partitions of $n$ with crank $m$, then the generating function is given by
$$
C(z,\tau) := \sum_{m,n=0}^\infty M(m,n)\zeta^mq^n = \prod_{n=1}^\infty \frac{1-q^n}{(1-\zeta q^n)(1-\zeta^{-1}q^n)},
$$
where $M(-1,1) = M(0,1) = M(1,1) = 1$,  $q = e^{2\pi i \tau}$, and $\zeta=e^{2\pi i z}$. Note that by setting $\zeta = 1$, we obtain the generating function for $p(n)$. Naturally, one may ask if the equidistribution of the crank (e.g. what Table \ref{table:rank} explains for the rank) can be seen solely from the crank generating function. Indeed, if $\ell$ is prime, then
$$\Phi_\ell(\zeta) \divides [q^{\ell n+\delta}]C(z,\tau) \implies p(\ell n+\delta) \equiv 0 \mod{\ell},$$ where $\Phi_{\ell}$ is the $\ell$-th cyclotomic polynomial, and $[q^{\ell n+\delta}]C(z,\tau)$ is the coefficient of $q^{\ell n+\delta}$ in $C(z,\tau)$ (see Lemma 3.1 in \cite{rolen2021cranks}). 
In this case, we say $C(z,\tau)$ \textit{explains the congruence} $p(\ell n+\delta)\inlinemod{0}{\ell}$. This terminology will be used throughout the paper. 

\begin{rem}
    It should be noted that $\Phi_\ell(\zeta) \divides [q^{\ell n+\delta}]C(z,\tau)$ is equivalent to saying the substitution $\zeta \to \zeta_\ell = e^{\frac{2 \pi i}{\ell}}$ results in $[q^{\ell n+\delta}]C(z,\tau)$ vanishing. This condition is what will be utilized later.
\end{rem}

In 2020, Larry Rolen, Zack Tripp, and Ian Wagner generalized $C(z,\tau)$ in order to develop a crank generating function for $p_k(n)$, the number of $k$-colored partitions of $n$ \cite{rolen2021cranks}. The generating function for $p_k(n)$ is given by
$$
\sum_{n=0}^\infty p_k(n)q^n = \prod_{n=1}^\infty \frac{1}{(1-q^n)^k},
$$
with a crank generating function given by
$$
C_k(\textbf{z},\tau) := C(0;\tau)^{\lfloor\frac{k}{2}\rfloor}\prod_{i=1}^{\lfloor\frac{k+1}{2}\rfloor}C(z_i,\tau).
$$
Here, $\textbf{z} = (z_1,z_2,\ldots,z_{\lfloor\frac{k+1}{2}\rfloor})$ and $z_i = a_iz$ where the $a_i$ are integers which depend on $k$. Note that for $1$-colored partitions (choosing $a = 1$), we get 
$$
C_{1}(\textbf{z},\tau) = \prod_{n=1}^\infty \frac{1-q^n}{(1-\zeta q^n)(1-\zeta^{-1}q^n)}.
$$
In general,
$$
C_{k}(\textbf{z},\tau) = \prod_{n=1}^\infty \frac{1-\sigma_{\text{odd}}(k)q^n}{(1-\zeta^{\pm a_1} q^n)\cdot\ldots\cdot(1-\zeta^{\pm a_{\lfloor\frac{k+1}{2}\rfloor}}q^n)},
$$
where $\sigma_{\text{odd}}(k) = 1$ if $k$ is odd and $0$ otherwise. Here, we use the shorthand 
$$(1-\zeta^{\pm a_1} q^n) = (1-\zeta^{a_1} q^n)(1-\zeta^{-a_1} q^n).$$

The goal of this paper will be to further generalize the crank defined by Rolen, Tripp, and Wagner in order to accommodate $(k+j)-$colored partitions where $j$ of the colors have distinct parts. To that end, let $p_{k,j}(n)$ denote the number of such  partitions of $n$, i.e.
$$
\sum_{n=0}^\infty p_{k,j}(n)q^n = \prod_{n=1}^\infty \frac{(1+q^n)^j}{(1-q^n)^k}.
$$
Through computational experimentation, it can be seen that these partitions likely satisfy a wide range of Ramanujan-like congruences. For example, using a computer algebra system, one can quickly see that if $k=4$ and $j=2$, it is likely that $p_{4,2}(5n+2) \equiv p_{4,2}(5n+4) \inlinemod{0}{5}$. In fact, these congruence will be proven in Section 3 by Theorem \ref{thm:mainj2}. This theorem and the two following it are given below as our main theorems.

\begin{thm}\label{thm:mainj2}
    Let $j=2$, and $k = \ell m+(\ell-1)$ where $m$ is a non-negative integer and $\ell$ is a prime. Then
    $$
    C_{k,2}(\textbf{z},\tau) = \prod_{n=1}^\infty\frac{(1+\zeta^{\pm2}q^n)}{(1-q^n)^m[(1-\zeta^{\pm1}q^n)(1-\zeta^{\pm3}q^n)\cdot\ldots\cdot(1-\zeta^{\pm\lfloor\frac{\ell}{2}\rfloor}q^n)]^{m+1}}
    $$
    defines a crank generating function for $p_{k,2}(n)$ which explains the Ramanujan-like congruences $$p_{k,2}(\ell n+\delta) \equiv 0 \mod{\ell},$$ where $8\delta+1$ is a quadratic non-residue modulo $\ell$.
\end{thm}

\begin{thm}\label{cor:j3}
    Let $j=3$, and $k = \ell m+(\ell-3)$ where $m$ is a non-negative integer and $\ell$ is a prime. Then
    $$
    C_{k,3}(\textbf{z},\tau) := \prod_{n=1}^\infty\frac{(1+q^n)(1+\zeta^{\pm(\ell-2)}q^n)}{[(1-q^n)(1-\zeta^{\pm(\ell-2)}q^n)]^m[(1-\zeta^{\pm1}q^n)\cdot\ldots\cdot(1-\zeta^{\pm(\ell-4)}q^n)]^{m+1}}
    $$
defines a crank generating function for $p_{k,3}(n)$ which explains the Ramanujan-like congruences $$p_{k,3}(\ell n+\delta)(n) \equiv 0 \mod{\ell},$$ where $4\delta+1$ is a quadratic non-residue modulo $\ell$.
\end{thm}

\begin{thm}\label{thm:mainjp}
    Let $j=\ell$, and $k = \ell m+(\ell-3)$ where $m$ is a non-negative integer and $\ell$ is a prime. Then 
   $$
    C_{k,j}(\textbf{z},\tau) = \prod_{n=1}^\infty\frac{(1+q^n)(1+\zeta^{\pm2}q^n)\cdot \ldots \cdot(1+\zeta^{\pm(\ell-1)}q^n)}{[(1-q^n)(1-\zeta^{\pm(\ell-2)}q^n)]^m[(1-q^n)(1-\zeta^{\pm1}q^n)\cdot \ldots \cdot(1-\zeta^{\pm(\ell-4)}q^n)]^{m+1}}
   $$
defines a crank generating function for $p_{k,j}$, explaining the Ramanujan-like congruences 
$$p_{k,j}(\ell n+\delta) \equiv 0 \mod{\ell},$$ where $8\delta+1$ is a quadratic non-residue mod $\ell$.
\end{thm}

\begin{rem}
    As it turns out, these crank generating functions are by no means unique. In fact, one may substitute the exponents in the denominator with any complete set of residues mod $\ell$. The result would be a different crank generating function that explains the same Ramanujan-like congruences.
\end{rem}

\section{Preliminary facts and definitions}

Throughout the paper, we will utilize Dedekind's eta function and Jacobi's theta functions. The definitions for these are given below.

\begin{defn}\label{defn:eta}
    Let $q = e^{2\pi i \tau}$. Then Dedekind's eta function is defined as
    $$
    \eta(\tau) = q^{1/24}\prod_{n=1}^\infty (1-q^n).
    $$
\end{defn}

\begin{defn}\label{defn:thetafncts}
    Let $\zeta = e^{2\pi i z}$ and $q = e^{2\pi i \tau}$. Jacobi's main theta function is defined as follows.
    $$
    \theta(z,\tau) := \sum_{n=-\infty}^\infty \left( \frac{-4}{n}\right)\zeta^\frac{n}{2}q^\frac{n^2}{8} = q^\frac{1}{8}(\zeta^\frac{1}{2}-\zeta^{-\frac{1}{2}})\prod_{n=1}^\infty(1-q^n)(1-\zeta q^n)(1-\zeta^{-1}q^n)
    $$
    We also define Jacobi's three auxiliary theta functions. Namely,
    \begin{align*}
        \theta_{01}(z,\tau) &:= \theta(z+\frac{1}{2};\tau) \\
        \theta_{10}(z,\tau) &:= q^\frac{1}{8}\cdot\zeta \cdot \theta(z+\frac{1}{2};\tau) \\
        \theta_{11}(z,\tau) &:= q^\frac{1}{8}\cdot\zeta^\frac{3}{4} \cdot \theta(z+\frac{1}{2}+\frac{\tau}{2};\tau).
    \end{align*}
\end{defn}
It is worth noting that if we alter the definitions slightly, we get a nice theta quotient for each auxiliary function. From \cite{Wagner2022JacobiFW} for example, we have
\begin{align*}
    \hat{\theta}_{01}(z,\tau) &:= -i\theta(z+\frac{1}{2};\tau) = \frac{\eta(\tau)^2\theta(2z,2\tau)}{\eta(2\tau)\theta(z,\tau)} \\
    \hat{\theta}_{10}(z,\tau) &:= -q^\frac{1}{8}\cdot\zeta^\frac{1}{2} \cdot \theta(z+\frac{\tau}{2};\tau) = \frac{\eta(\tau)^2\theta(z,\frac{\tau}{2})}{\eta(\frac{\tau}{2})\theta(z,\tau)} \\
    \hat{\theta}_{11}(z,\tau) &:= -iq^\frac{1}{8}\cdot\zeta^\frac{1}{2} \cdot\theta(z+\frac{1}{2}+\frac{\tau}{2};\tau) = \frac{\eta(\frac{\tau}{2})\eta(2\tau)\theta(2z,\tau)\theta(2z,\tau)\theta(z,\tau)}{\eta(\tau)\theta(z,\frac{\tau}{2})\theta(2z,2\tau)}.
\end{align*}
For this paper, we are mainly interested in $\theta(z,\tau)$ without the factors of $q^\frac{1}{8}(\zeta^\frac{1}{2}-\zeta^{-\frac{1}{2}})$ in front of the product. Thus, let $\Tilde{\theta}(z,\tau)$ be $\theta(z,\tau)$ with those factors removed, and similarly for the auxiliary functions. Further, we will only need $\Tilde{\theta}(z,\tau)$ and $\Tilde{\theta}_{01}(z,\tau)$, although the other auxiliary theta functions are listed since they could be used to define different cranks for $p_{k,j}(n)$. \\

\begin{rem}
    In \cite{rolen2021cranks}, the authors use the theory of theta blocks developed by Gritsenko, Skoruppa, and Zagier \cite{Gritsenko2019ThetaB} to define their crank for all positive integers, $k$. Here, we use the same techniques, but we avoid utilizing the full list of weight 1 theta blocks. This simplifies the proofs, but restricts which $k$ and $j$ we may have. The author conjectures a crank generating function for all $k$ and $j$ which are positive integers at the end of this paper. 
\end{rem}

In the next section, we will utilize the following lemmas.
\begin{lem}\label{lem:triangle}
If $8\delta+1$ is a quadratic non-residue $(\!\!\!\!\mod{\ell})$ then $\ell n+\delta$ is not triangular for all $n$.
\end{lem}

\begin{proof}
    We proceed by contraposition. Suppose $\ell n+\delta$ is triangular. Then there exists a positive integer $k$ such that 
    $$
    \ell n+\delta = \frac{k(k+1)}{2}.
    $$
    Rearranging the equation and utilizing the quadratic formula, we find that
    $$
    k = \frac{-1\pm\sqrt{1-4(1)(-2\ell n-2\delta)}}{2}.
    $$
    Since $k$ is an integer, it follows that $8\delta+1$ must be a perfect square mod $\ell$.
\end{proof}

When finding Ramanujan-like congruences mod $\ell$, it is often useful to view the generating function as a product of one generating function multiplied by a power series in $q^\ell$. Doing so allows us to simplify the remaining generating function into something more manageable. The following lemma shows that such congruences mod $\ell$ are preserved under this factorization, which allows us to ignore power series in $q^\ell$ when proving our main theorems.

\begin{lem}[Proposition 3 in \cite{Ono_1996}]\label{lem:congprod}
Let $m$ be a positive integer, $A := \sum_{n=0}^\infty a(n)q^{n}$, and $B := 1+\sum_{n=1}^\infty b(\ell n)q^{\ell n}$. Define $C := AB = \sum_{n=0}^\infty c(n)q^n$ and $D := AB^{-1} = \sum_{n=0}^\infty d(n)q^n$. Suppose $A$ satisfies the Ramanujan-like congruence $a(\ell n+\delta) \inlinemod{0}{\ell}$. Then $c(\ell n+\delta) \equiv d(\ell n+\delta) \inlinemod{0}{\ell}$.
\end{lem}

\section{Main Theorems}
We are now equipped to prove Theorems \ref{thm:mainj2}, \ref{cor:j3} and $\ref{thm:mainjp}$.

\begin{proof}[Proof of Theorem \ref{thm:mainj2}]
Note that if we multiply our proposed crank generating function by $1 = \frac{\eta(\tau)}{\eta(\tau)} = \frac{\prod_{n=1}^\infty(1-q^n)}{\prod_{n=1}^\infty(1-q^n)}$, then we obtain
$$
C_{k,2}(\textbf{z},\tau) = \prod_{n=1}^\infty\frac{(1-q^n)(1+\zeta^{\pm2}q^n)}{(1-q^n)^{m+1}[(1-\zeta^{\pm1}q^n)(1-\zeta^{\pm3}q^n)\cdot\ldots\cdot(1-\zeta^{\pm\lfloor\frac{\ell}{2}\rfloor}q^n)]^{m+1}}.
$$
Note the denominator is equal to 
$$
\prod_{n=1}^\infty[1-q^{\ell n}+\Phi_\ell(\zeta)r(z,\tau)]^{m+1},
$$ 
where $r(\tau)$ is a polynomial in $q$ with coefficients in $\mathbb{Q}[\zeta,\zeta^{-1}]$. If we set $\zeta = \zeta_\ell$, a primitive $\ell$-th root of unity, then the denominator becomes equivalent to a power series in $q^{\ell}$. Hence by Lemma \ref{lem:congprod}, it suffices for the numerator to satisfy the relevant congruences. Since we multiplied by $\prod_{n=1}^\infty(1-q^n)$, the numerator becomes $\Tilde{\theta}(2z+\frac{1}{2},\tau)$, which only has triangular powers \cite{NIST:DLMF}. Thus, given any arithmetic progression $\ell n+\delta$ that is disjoint from the triangle numbers, we immediately obtain

$$\Phi_\ell(\zeta) \divides [q^{\ell n+\delta}]C_{k,j}(\textbf{z},\tau).$$ 

By Lemma \ref{lem:triangle}, it suffices that $8\delta+1$ is a quadratic non-residue mod $\ell$.
\end{proof}

\begin{proof}[Proof of Theorem \ref{cor:j3}]
The proof is analogous to that of Theorem \ref{thm:mainj2} if one multiplies by $$1=\frac{\prod_{n=1}^\infty(1+q^n)(1+\zeta^{(\ell-2)}q^n)(1+\zeta^{-(\ell-2)}q^n)}{\prod_{n=1}^\infty(1+q^n)(1+\zeta^{(\ell-2)}q^n)(1+\zeta^{-(\ell-2)}q^n)},$$ instead of 

$$
1 = \frac{\prod_{n=1}^\infty (1-q^n)}{\prod_{n=1}^\infty (1-q^n)}
$$
    
\end{proof}
We now prove Theorem \ref{thm:mainjp} a similar method.

\begin{proof}[Proof of Theorem \ref{thm:mainjp}]
    We first multiply the crank generating function by 
    $$
    1 = \prod_{n=1}^\infty\frac{(1-q^n)(1-\zeta^{\pm2}q^n)\cdot \ldots \cdot(1-\zeta^{\pm\ell-1}q^n)(1-q^n)(1-\zeta^{\pm(\ell-2)}q^n)}{(1-q^n)(1-\zeta^{\pm2}q^n)\cdot \ldots \cdot(1-\zeta^{\pm\ell-1}q^n)(1-q^n)(1-\zeta^{\pm(\ell-2)}q^n)}.
    $$ 
    Note that when we do so, the numerator is equal to
    $$
    \prod_{n=1}^\infty[1-q^{2\ell n}+\Phi_{\ell}(\zeta)r_1(\tau)]\cdot \Tilde{\theta}((\ell-2)z,\tau),
    $$
     where $r_1(\tau)$ is a polynomial in $q$. Also the denominator becomes
    $$
    \prod_{n=1}^\infty[1-q^{\ell n}+\Phi_\ell(\zeta)r_2(\tau)].
    $$
    Now, we set $\zeta = \zeta_\ell$. Doing so yields
    $$
    \prod_{n=1}^\infty\left[\frac{(1-q^{2\ell n})}{(1-q^{\ell n})^{m+2}}\right] \cdot \Tilde{\theta}((\ell-2)z,\tau).
    $$

    By Lemma \ref{lem:congprod}, it suffices to only consider the powers of $q$ in $\Tilde{\theta}((\ell-2)z,\tau)$, all of which are triangular \cite{NIST:DLMF}. Thus by Lemma \ref{lem:triangle}, it suffices that $8\delta+1$ is a quadratic non-residue mod $p$.
\end{proof}

\section{Other Problems, Conjectures, and Conclusions}
We conclude with a few observations. It can be shown that $p_{k,j}(n)$ satisfies more Ramanujan-like congruences than what are currently covered by the above arguments. It is likely possible to generalize the crank generating functions defined here in order to cover most (if not all) such congruences satisfied by $p_{k,j}(n)$ for all $k,j$ which are positive integers. In Section 2 , we mentioned that the theory of theta blocks is utilized to provide proofs of the crank generating functions for $p_k(n)$. The author suspects that, as in \cite{rolen2021cranks}, the theory can be used to prove the following conjecture.

\begin{con}\label{con:bigcrank}
Assume that $k,j$ are positive odd integers. Then
$$
C_{k,j}(\textbf{z},\tau) := \prod_{n=1}^\infty\frac{(1-q^n)(1+q^n)(1+\zeta^{\pm a_1}q^n)(1+\zeta^{\pm a_2}q^n)\cdot\ldots\cdot(1+\zeta^{\pm a_{\frac{j-1}{2}}}q^n)}{(1-\zeta^{\pm b_1}q^n)(1-\zeta^{\pm b_2})\cdot\ldots\cdot(1-\zeta^{\pm b_{\frac{k+1}{2}}}q^n)}
$$
(such that the $a_i$ are consecutive even integers and $b_i$ are consecutive odd integers) defines a crank for the desired partitions which explains most of the partition's Ramanujan-like congruences.
\end{con}

\begin{rem}
    The choice of the $a_i$ and $b_i$ are not unique in the fact that other choices may also result in a valid crank generating function. For even $k$, one only needs to remove the factor of $(1-q^n)$ in the numerator. For even $j$, one should remove the factor of $(1+q^n)$ in the numerator, then switch $a_{\frac{j-1}{2}}$ to $a_{\frac{j+1}{2}}$
\end{rem}

In addition, the author is also interested in answers to the following questions.

\begin{enumerate}
    \item Assuming the above defines a crank for $p_{k,j}$, what are the ``singular congruences" (defined by Rolen, Tripp, and Wagner in \cite{rolen2021cranks}) of this crank?
    \item Is there a nice interpretation of the generating function which allows us to know how to explicitly calculate a crank? (See Section 12.4 of \cite{HIRSCHHORN_2018} for more details.)
    \item Can/How can theta blocks be used to prove the above crank generating function, as well as others?
\end{enumerate}

\section{Acknowledgements}
The author would like to thank Marie Jameson for her continued guidance and advice throughout the writing process. The author also thanks Dalen Dockery for support in editing and useful insights.

\bibliographystyle{alpha}
\bibliography{refs}

\end{document}